\documentclass[11pt]{amsart}

\usepackage{amsthm}
\usepackage{mathptmx}  

\usepackage{amsmath,amssymb,amscd,amsfonts,setspace}
\usepackage{graphicx,lscape}    
\usepackage{hyperref} 
\usepackage[vcentermath]{youngtab}    
\usepackage[centertableaux]{ytableau}

\newcommand{\headingstyle}[1]{\noindent\textbf{#1}}

\newcommand{\ds}{\displaystyle}

\newcommand{\sub}{\Yboxdim4pt}

\newcommand{\CC}{\mathbb{C}}
\newcommand{\QQ}{\mathbb{Q}}
\newcommand{\RR}{\mathbb{R}}
\newcommand{\ZZ}{\mathbb{Z}}

\DeclareMathOperator{\codim}{codim}

\DeclareMathOperator{\lcm}{lcm}
\DeclareMathOperator{\order}{order}

\newdimen\CellWidth
\newdimen\CellHeight
\newdimen\LineThickness

\def\CharacterCell(#1){%
	\hbox to\CellWidth{\vrule width\LineThickness\hfil%
		\vbox to\CellHeight{\vfil\hbox{#1}\vfil}\hfil}}

\def\TableauRowRecursive(#1,#2){%
	\hbox{\if#2,\vbox{\CharacterCell(#1)}%
		\else\vbox{\CharacterCell(#1)}\vbox{\TableauRowRecursive(#2)}\fi}}

\def\TableauRow(#1){%
	\hbox{\vbox{\TableauRowRecursive(#1,,) %
			\hrule height \LineThickness depth 0.0pt}%
		\vrule width \LineThickness}}

\def\TableauRecursive[#1;#2]{%
	\vbox{\if#2;\TableauRow(#1)%
		\else\TableauRow(#1)\hbox{\TableauRecursive[#2]}\fi}}

\newtheorem{theorem}{Theorem}[section]
\newtheorem{corollary}[theorem]{Corollary}
\newtheorem{lemma}[theorem]{Lemma}

\theoremstyle{definition}
\newtheorem{definition}[theorem]{Definition}
\newtheorem{example}[theorem]{Example}

\theoremstyle{remark}
\newtheorem{remark}[theorem]{Remark}

\numberwithin{equation}{section}
\numberwithin{table}{section}

\begin{document}

\title[Spectra complex reflection groups]
{Spectra of Cayley graphs
	of complex reflection groups}
\author{Briana Foster-Greenwood}
\address{Department of Mathematics \\
	Idaho State University \\
	Pocatello, ID 83209-8085}
\email{fostbria@isu.edu}
\author{Cathy Kriloff}
\address{Department of Mathematics \\
	Idaho State University \\
	Pocatello, ID 83209-8085}
\email{krilcath@isu.edu}

\begin{abstract}
Renteln proved that the eigenvalues of the distance matrix of a Cayley graph of a real reflection group with respect to the set of all reflections are integral and provided a combinatorial formula for some such spectra.
We prove the eigenvalues of the distance, adjacency, and codimension matrices of Cayley graphs of complex reflection groups with connection sets consisting of all reflections are integral and provide a combinatorial formula for the codimension spectra for a family of monomial complex reflection groups.
\end{abstract}
        

\maketitle
\section{Introduction}
\label{sec:intro}
A Cayley graph of a finite group $G$ has vertices given by the elements in $G$ and an edge joining $g, h \in G$ whenever $hg^{-1}$ is in a specified set $T\subseteq G\setminus\{1\}$ generating $G$ and closed under inversion.\footnote{This defines a left Cayley graph and the conditions on $T$ imply it has no loops, is connected, and is undirected.}
For a complex-valued function $f$ on $G$ and an ordering of the elements of $G$, one can define the matrix with rows and columns indexed by group elements and $(g,h)$-entry $f(gh^{-1})$.
When $f$ is the indicator function of $T$ this yields the adjacency matrix of the Cayley graph,
and when $f$ is the length function with respect to $T$ this yields the distance matrix that records the lengths of shortest paths between vertices.

A real (respectively complex) reflection group is a finite group generated by reflections of a Euclidean (respectively unitary) vector space.
A reflection group is irreducible if that reflection representation is irreducible.
The classifications of real and complex irreducible reflection groups 
consist of one or more infinite families and finitely many exceptional cases.
In~\cite{Renteln11} it is shown that the spectrum of the distance matrix of a Cayley graph of an irreducible real reflection group with $T$ the set of all reflections is integral.
Our main result, Theorem~\ref{th:distintegrality}, extends this to complex reflection groups.

\textbf{Theorem.}
\textit{The spectrum of the distance matrix of a Cayley graph of an irreducible complex reflection group with respect to the set of all reflections is integral.}

The methods used in~\cite{Renteln11} depend on the well-known result that for real reflection groups the reflection length function and the function giving the codimension of the subspace fixed by a particular group element are the same~\cite{Carter72}.
For complex reflection groups these functions can differ~\cite{Shi07a,Shepler-Witherspoon11}, and work of the first author~\cite{Foster-Greenwood13} completes the characterization of when this occurs.
The result in~\cite{Renteln11} also uses that the codimension function is constant on rational conjugacy classes, which is not obvious for the reflection length function.

For the infinite family of complex reflection groups our approach is to consider reflection preserving automorphisms to show that in fact the reflection length function is also constant on rational conjugacy classes.  
This appears in Section~\ref{sec:integrality} following some preliminaries on reflection groups, Cayley graphs, and spectra in Sections~\ref{sec:refgroups}, \ref{sec:Cayleygraphs}, and~\ref{sec:charformulas}.  
For the exceptional groups we utilize computer calculations in GAP3~\cite{gap3}\nocite{CHEVIE} because analysis of the automorphisms is more subtle in this case and appears to involve computer calculations as well\textemdash see~\cite{Marin-Michel10}.
In Section~\ref{sec:combformula} we review the theory of symmetric functions necessary to provide a combinatorial formula (Theorem~\ref{th:codimspec}) for codimension spectra for a family of the monomial complex reflection groups analogous to that given for distance, equivalently codimension, spectra for real reflection groups in~\cite{Renteln11}.

We were motivated by the natural question, common in representation theory, of whether results known for real reflection groups can be extended to complex reflection groups, often with the goal of better understanding these more general groups.
However, questions related to integral adjacency spectra date back to~\cite{Harary-Schwenk74} and have been explored extensively for trees and other graphs with special vertex properties (see~\cite{BCRSS02}) as well as for Cayley graphs of abelian groups\textemdash in particular for circulant graphs, where the group is cyclic~\cite{Bridges-Mena82,So06,Alperin-Peterson12}\textemdash and for nonabelian groups with nicely behaved connection sets in~\cite{DKMA13,Godsil-Spiga14}.
Two significant classification results regarding groups that have Cayley graphs with integral adjacency spectrum appear in~\cite{ABM14}.
There are also contributions relating adjacency and distance spectra of Cayley graphs of cyclic groups~\cite{Ilic10}, and more generally abelian groups, even for a broad class of connection sets~\cite{Klotz-Sander12}.

\section{Background on reflection groups}
\label{sec:refgroups}

A real reflection group is a finite group generated by orthogonal reflections of a real vector space (see~\cite{Bourbaki02,Humphreys90}).
It can be shown that there is a natural set of generating reflections up to conjugacy, so such a choice is fixed and these are termed \textit{simple reflections}.
There is a classification of irreducible real reflection groups, using diagrams that graphically encode the simple reflections and the orders of their products, into four infinite families and six exceptional groups.
Those groups that are crystallographic, i.e., preserve an integral lattice, are the Weyl groups that arise in Lie theory, so a reflection group is often denoted by $W$ (even when not crystallographic or not real) as we will do here.

The notion of reflection and the classification of finite groups generated by reflections extend to the setting of an $n$-dimensional complex vector space $V$ (for a brief survey, see~\cite{Geck-Malle06} and for more details see~\cite{Lehrer-Taylor09}).
A linear transformation on $V$ is a \textit{reflection} if it is of finite order and has an $(n-1)$-dimensional eigenspace corresponding to eigenvalue $1$.
In the case of a real reflection, the remaining $1$-dimensional space is an eigenspace corresponding to eigenvalue $-1$, but in the remaining complex dimension a complex reflection acts by a root of unity that may have order greater than two.
A finite subgroup $G$ of $\mathrm{GL}(V)$ generated by reflections is called a \textit{reflection group on $V$}.
Since $G$ is finite, the standard averaging technique makes it possible to fix a non-degenerate $G$-invariant hermitian form on $V$ and consider $G$ as a subgroup of the unitary group on $V$.
Finiteness of $G$ also guarantees the representation on $V$ is completely reducible, which means it suffices to consider reflection groups and spaces on which they act irreducibly.
More precisely, $G$ is said to \textit{act irreducibly in dimension $k$} if its fixed point space is of dimension $n-k$ and it acts irreducibly when restricted to the complement of that fixed point space. 

We describe an infinite family of complex reflection groups.
Let $r,p,n \geq 1$ with $p$ dividing $r$, and let $\zeta$ be a primitive $r$-th root of unity.
Under the \textit{standard monomial representation}, the group $G(r,p,n)$ consists of 
\begin{align*}
 &\hbox{monomial matrices with nonzero entries $\zeta^{a_1},\dots, \zeta^{a_n}$ such that} \\
 &\hbox{$(\zeta^{a_1}\cdots \zeta^{a_n})^{r/p}=1$, or equivalently $a_1+\cdots +a_n \equiv 0 \bmod{p}$.}
\end{align*}
Each such monomial matrix may be written as a product of a diagonal matrix with entries $\zeta^{a_1},\dots,\zeta^{a_n}$ and a permutation matrix (obtained by permuting columns of the identity matrix).
Keeping track of only the exponents provides an alternative description of $G(r,p,n)$ as an index $p$ subgroup of $\ZZ_r\wr S_n=\ZZ_r^n \rtimes S_n$ and makes it clear that $|G(r,p,n)|=r^n n!/p$.
In this perspective, $G(r,p,n)$ consists of all
\[(a_1,\dots,a_n \mid \sigma) \hbox{ such that } a_i \in \ZZ_r,\, a_1+\cdots +a_n \equiv 0 \bmod{p}, \hbox{ and } \sigma \in S_n,\]
and the action of $S_n$ on $\ZZ_r^n$ providing the semidirect product structure on $G(r,p,n)$ is
\[\sigma.(a_1,\dots,a_n)=(a_{\sigma^{-1}(1)},\dots,a_{\sigma^{-1}(n)}),\]
so that if $\sigma, \tau \in S_n$,
\[(a_1,\dots,a_n \mid \sigma)(b_1,\dots,b_n \mid \tau)=(a_1+b_{\sigma^{-1}(1)},\dots,a_n+b_{\sigma^{-1}(n)} \mid \sigma\tau).\]
This is consistent with matrix multiplication where $(a_1, a_2, \dots , a_n \mid \sigma)$ represents the $n \times n$ matrix whose only nonzero entries are the $\zeta^{a_i}$ in position $(i,\sigma^{-1}(i))$ for $1\leq i\leq n$. In matrix form, each element of $G(r,p,n)$
is conjugate by a permutation to a direct sum of blocks $\delta_i\sigma_i$,
where $\delta_i$ is diagonal and $\sigma_i$ is a cyclic permutation.
For a block $\delta_i\sigma_i$, define the \textit{cycle-size} to be the order of $\sigma_i$ and
the \textit{cycle-product} to be the determinant of $\delta_i$.
In Section~\ref{sec:integrality} it will be more convenient to work additively, so
for a cycle-product $\det(\delta_i)=\zeta^{c_i}$, define the corresponding \textit{cycle-sum} to be
$c_i \bmod{r}$.
Two elements of $G(r,1,n)$ are conjugate if and only if they have the same
multiset of (cycle-size, cycle-product) pairs. Conjugacy classes may split upon
restriction to the normal subgroup $G(r,p,n)$.

A reflection group $G$ on $V$ is \textit{imprimitive} if there is a decomposition $V=V_1\oplus V_2 \oplus \cdots \oplus V_k$ into proper nonzero subspaces such that $G$ permutes the subspaces.
The groups $G(r,p,n)$ with $r,n\geq 2$ are imprimitive in their action on a subset of the set of lines orthogonal to the reflecting hyperplanes and the exceptional groups are primitive.
The symmetric groups $G(1,1,n)$ do not act irreducibly in the standard monomial representation but do act irreducibly on the complement of the span of the sum of all the basis vectors and are primitive on this $(n-1)$-dimensional subspace. 

The irreducible finite reflection groups were classified in~\cite{Shephard-Todd54} (see also~\cite{Lehrer-Taylor09}) and consist of the
\begin{itemize}
\item groups $G(r,p,n)$, with $r\geq 2$, $n\geq 1$, and $(r,p,n)\neq (2,2,2)$, which are imprimitive and irreducible in dimension $n$, 
\item symmetric groups $G(1,1,n)$, which are primitive and irreducible in dimension $n-1$, and
\item $34$ primitive exceptional groups, numbered $G_4,\dots,G_{37}$, irreducible in dimensions $2$ to $8$.
\end{itemize}

The reason for the numbering is that the original classification listed $G(r,1,1)$, $G(1,1,n)$, and $G(r,p,n)$ with $r,n \geq 2$ separately.
It is common to refer to all $G(r,p,n)$ with $n\geq 1$ and $r>1$ as the monomial groups.
If there is a $G$-invariant real subspace, $V_0$, of $V$ so that the canonical map $\CC \otimes_{\RR} V_0 \to V$ is a bijection, then $G$ is a \textit{real reflection group}.
The finite real reflection groups occur in the classification from~\cite{Shephard-Todd54} as:
\begin{align*}
G(1,1,n), &\hbox{ type } A_{n-1}, \hbox{ the symmetric group } S_n; \\
G(2,1,n), &\hbox{ type } B_n, \hbox{ the hyperoctahedral group } \ZZ_2^n \rtimes S_n; \\
G(2,2,n), &\hbox{ type } D_n, \hbox{ an index two subgroup of } \ZZ_2^n \rtimes S_n; \\
G(r,r,2), &\hbox{ type } I_2(r), \hbox{ the dihedral group of order } 2r; \hbox{ and} \\
G_{23}, G_{28}, G_{30}&, G_{35}, G_{36}, \hbox{ and } G_{37}, \hbox{ types } H_3, F_4, H_4, E_6, E_7, \hbox{ and } E_8 \hbox{ respectively}.
\end{align*}

For general complex reflection groups there is not a choice of generators that is canonical up to conjugacy as with simple reflections in Coxeter groups.  
Commonly used presentations appear in~\cite{BMR98}, matrix forms of generators for the monomial groups are provided in~\cite{Lehrer-Taylor09}, and various other presentations are analyzed in~\cite{Shi07b}. 

\section{Background on Cayley graphs and spectra}
\label{sec:Cayleygraphs}

For a finite group $G$ and a subset $C$ of $G\setminus\{1\}$ with $C=C^{-1}$, 
the (left) \textit{Cayley graph}, $\Gamma(G,C)$, of $G$ with respect to $C$ has vertices
corresponding to the elements $g \in G$ and an edge joining $g$ and $cg$
for each $g \in G$ and $c \in C$.  In particular, $\Gamma(G,C)$ is an undirected
graph with no loops and no multiple edges.
The Cayley graph is always vertex-transitive and regular and is connected when $C$ generates $G$.
If $C$ is closed under conjugation then a corresponding right Cayley graph with an edge joining $g$ and $gc$ is isomorphic to the left Cayley graph since $gc=(gcg^{-1})g$.

\begin{definition}
For a graph $\Gamma=\Gamma(G,C)$ with a choice of ordering of the vertices of $\Gamma$ the \textit{adjacency matrix} $A_\Gamma$ has $(g,h)$-entry 
\[A_\Gamma(g,h)=\hbox{ the number of edges joining vertices $g$ and $h$,}
\]
and the \textit{distance matrix} $D_\Gamma$ has $(g,h)$-entry
\[D_\Gamma(g,h)=\hbox{ the minimal length of a path from $g$ to $h$ in $\Gamma$}. \]
Since $A_{\Gamma}$ and $D_{\Gamma}$ are symmetric, they have real eigenvalues.
If the \textit{spectrum} (set of eigenvalues) of the matrix $A_{\Gamma}$ ($D_{\Gamma}$) consists entirely of integers,
then the graph $\Gamma$ is \textit{adjacency} (\textit{distance}) \textit{integral}.
\end{definition}

Since right multiplication is an automorphism of the left Cayley graph,
\begin{align*}
A_\Gamma(g,h)&=A_\Gamma(gh^{-1},1)=A_\Gamma(1,gh^{-1}), \hbox{ and} \\
D_\Gamma(g,h)&=D_\Gamma(gh^{-1},1)=D_\Gamma(1,gh^{-1}).
\end{align*}

A crucial observation related to the approach in Section~\ref{sec:charformulas} is that 
\begin{equation}
\label{eq:a}
A_\Gamma(1,gh^{-1})=\delta_C(gh^{-1}),  \hbox{ for the indicator function }  
\delta_C(x)=\begin{cases}
     1 & \text{ if $x\in C$}, \\
     0 & \text{ if $x\notin C$}
\end{cases}
\end{equation}
and
\begin{equation}
\label{eq:d}
D_\Gamma(1,gh^{-1})=\ell_C(gh^{-1}),  \hbox{ for }  
\ell_C(x)=\begin{cases}
     \min\{k \mid x=c_1\cdots c_k \text{ for } c_i \in C\} & \text{ if $x\neq 1$}, \\ 
     0 & \text{ if $x=1$}
\end{cases}.
\end{equation}

Other commonly studied matrices associated to a graph are the Laplacian matrix, $L_\Gamma^-=V_\Gamma-A_\Gamma$, and signless Laplacian, $L_\Gamma^+=V_\Gamma+A_\Gamma$, where $V_\Gamma$ is diagonal with $(u,u)$-entry $V_\Gamma(u,u)=$ the degree of vertex $u$. 
Since $\Gamma(G,C)$ is regular of degree $d=|C|$, then $L_\Gamma^\pm=dI\pm A_\Gamma$ and the spectrum of $L_\Gamma^\pm$ is just a translation of the spectrum of $A_\Gamma$.  
Similarly, the normalized Laplace and Seidel spectra are easily related to the spectrum of $A_\Gamma$ when $\Gamma$ is regular. 
Hence we do not consider any Laplace or Seidel spectra here.

Our focus is on Cayley graphs $\Gamma(W,T)$ where 
\begin{center}
$W$ is a complex reflection group and $T$ is the set of all reflections of $W$.
\end{center}
Note that $T$ generates $G$ and is closed under conjugation. 
Consequently, the left Cayley graph $\Gamma(W,T)$ is connected and isomorphic to a corresponding right Cayley graph.
We often abbreviate $\Gamma(W,T)$ as $\Gamma$.
The function $\ell_T$ has been called the \textit{reflection length} or \textit{absolute length} in~\cite{Shi07a,Armstrong09,Shepler-Witherspoon11} and is one example of word length considered in geometric group theory.  
Notice this is not the length function with respect to simple reflections more commonly used in Lie theory.

When $W$ is a real reflection group, the Cayley graph is the underlying graph for a partial order defined using $\ell_T$, called the absolute or reflection order.
When $W$ is a complex reflection group, this need no longer be the case.  
We briefly describe the partial order in order to emphasize the distinction.
The \textit{absolute} or \textit{reflection order} is defined by 
\[ u\leq_T w \hbox{ if and only if } \ell_T(u)+\ell_T(u^{-1}w)=\ell_T(w).\]
We denote by $(W,\leq_T)$ the reflection order poset and use $u \lessdot_T w$ to mean that $u<_T w$ and there does not exist $v \in W$ with $u <_T v <_T w$.

If $u\lessdot_{T} w$ is a covering relation in the poset $(W,\leq_{T})$, then
one can show $u^{-1}w$ must be a reflection.  The conjugate element
$wu^{-1}$ is also a reflection, so $(u,w)=(u,wu^{-1}u)$ is an edge
in the left Cayley graph $\Gamma(W,T)$.
Thus the poset $(W,\leq_T)$ is a subgraph of $\Gamma(W,T)$.
But $\Gamma(W,T)$ may contain additional edges as the next two examples illustrate.

\begin{example}
	Consider any complex reflection group $W$ in which there is a reflection $t$ of order three.  
	The Cayley graph $\Gamma(W,T)$ contains an edge joining the reflection $t$ to $t^2$,
	but $t^2$ is also a reflection, so $\ell_T(t)=\ell_T(t^2)=1$, and consequently, 
	the elements $t$ and $t^2$ are not comparable in the reflection length poset $(W, \leq_T)$.
\end{example}

\begin{example}
Consider the cyclic group of order $r$ as a complex reflection group, $G(r,1,1)$. 
Every $w\in G(r,1,1)$ with $w\neq 1$ is a reflection, so $\Gamma(G(r,1,1),T)$ is the complete graph $K_r$, which has $D_{K_r}=A_{K_r}$.  
It is well-known and easily checked that the spectrum is $\{r-1, -1^{r-1}\}$, where the power indicates multiplicity.  
In contrast, the graph of the reflection order poset $(G(r,1,1),\leq_T)$ is the $r$-star, $S_r$, which also has $D_{S_r}=A_{S_r}$, 
but with spectrum $\{-2^{r-2}, r-2\pm\sqrt{r^2-3r+3}\}$ (see~\cite[Proposition~9]{Trinajstic83}).
Note that $G(r,p,1)\simeq G(r/p,1,1)$ is included here.
\end{example}

The following theorem characterizes when the Cayley graph and reflection length poset coincide.  
The equivalence of (1) and (3) essentially appears in~\cite{Armstrong09} in the comments in the paragraph following Example~2.3.3 and the end of the paragraph following Definition~2.4.4.

\begin{theorem} Let $W$ be a complex reflection group.  The following are equivalent.
\label{th:graphvsposet}
\begin{enumerate}
\item The Cayley graph $\Gamma(W,T)$ is the underlying graph of the poset $(W,\leq_T)$.
\item The Cayley graph $\Gamma(W,T)$ is bipartite.
\item Every reflection in $W$ is of order $2$. 
\end{enumerate}
\end{theorem}
\begin{proof}
First assume $\Gamma(W,T)$ is the underlying graph of $(W,\leq_T)$.  
As mentioned after Definition~2.4 in~\cite{Armstrong09}, $(W,\leq_T)$ is a graded poset with $\ell_T$ as its rank function.  
Thus the vertices can be partitioned into those of even and odd rank, with none of common parity adjacent to each other.

Second assume that $\Gamma(W,T)$ is bipartite and let $t\in T$ with order $m>2$.  
Then $1$, $t$, and $t^2$ form a cycle of length $3$, contradicting that $\Gamma(W,T)$ has no cycles of odd length.  Thus every $t\in T$ has order $2$.

Third assume every reflection in $W$ is of order $2$ and that $\Gamma(W,T)$ is not the underlying graph of $(W,\leq_T)$.  
Then there is an edge $\{w,wt\}$ with $w$ and $wt$ not comparable in the reflection order.  This forces
\begin{align*}
\ell_T(w)+\ell_T(w^{-1}wt)&=\ell_T(w)+1>\ell_T(wt) \hbox{ and} \\
\ell_T(wt)+\ell_T(t^{-1}w^{-1}w)&=\ell_T(wt)+1>\ell_T(w).
\end{align*}
Hence $1>|\ell_T(w)-\ell_T(wt)|$, and since the right side is a nonnegative integer, $\ell_T(w)=\ell_T(wt)$.
Let $w=t_1\cdots t_k$ be a reduced expression for $w$.  
Each reflection $t_i$ is of order $2$, so in the geometric realization of $W$, $\det(t_i)=-1$ for $1\leq i\leq k$.  
But then $\det(w)=(-1)^k$ while $\det(wt)=(-1)^{k+1}$.
This contradiction proves that $\Gamma(W,T)$ and$(W,\leq_T)$ are the same graphs.
\end{proof}

There is a spectral criterion equivalent to the conditions in Theorem~\ref{th:graphvsposet}.
It is well-known and easy to prove that a graph is bipartite if and only if its adjacency spectrum is symmetric with respect to zero~\cite[Theorem~3.2.3]{CRS10}, and in particular a connected graph is bipartite if and only if its largest adjacency eigenvalue is the negative of its smallest adjacency eigenvalue~\cite[Theorem~3.2.4]{CRS10}.  

The first part of~\cite{Renteln11} shows that for $W$ a real reflection group, $\Gamma(W,T)$, or equivalently $(W,\leq_T)$, is distance integral and provides formulas for the distance spectrum 
in terms of the irreducible characters of $W$.
We do not consider spectra of order posets further here because, by Theorem~\ref{th:graphvsposet},
$\Gamma(W,T)$ is not the underlying graph of $(W,\leq_T)$ for many complex reflection groups,
and when $W=G(3,1,2)$ or $W=G_4$ for example, hand and computer calculations indicate that
the distance spectrum of $(W,\leq_T)$ is non-integral.

The second part of~\cite{Renteln11} considers the Cayley graph $\Gamma(W,S)$,
where $W$ is a real reflection group and $S$ is the set of simple reflections,
and shows that $\Gamma(W,S)$ is distance integral when $W$ is of type $A$, $D$, or $E$ (i.e., a simply-laced reflection group). 
For each complex reflection group, there is a set $S$ of generators given in the standard references~\cite{BMR98,Lehrer-Taylor09}, but these are not as well-behaved as the set of simple reflections for the real groups\textemdash see~\cite{Bremke-Malle98,Shi02}.
For such $S$ the graph $\Gamma(W,S)$ does not appear to have integral adjacency or distance spectra in general.  
We used a mixture of symbolic and numerical computer calculations in Sage~\cite{sage} calling GAP3~\cite{gap3} and Mathematica to compute these spectra for several $\Gamma(W,S)$.
These calculations indicate that for the $19$ exceptional complex reflection groups with order less than $1000$, the adjacency and distance spectra of $\Gamma(W,S)$ are non-integral.
They also indicate that for twenty relatively small groups of the form $G(r,p,n)$ the adjacency and distance spectra are non-integral (i.e., contain some non-integers) in all cases except $G(2,2,n)$ (which are real of type $D_n$ and already treated in~\cite{Renteln11}) and $G(6,3,2)$ (of order $24$ but not isomorphic to $G(2,2,3)$).  
For these reasons we do not consider the spectra of the graphs $\Gamma(W,S)$ further here.

\section{Character formulas and integrality}
\label{sec:charformulas}

In this section we review what is known regarding character formulas and integrality in a general context and include an aside on the spectral radius.

\headingstyle{Character Formulas}\\
The following definition and lemma appear in~\cite{Diaconis-Shahshahani81}.  
Similar perspectives that apply to any finite group are explained in~\cite{Cesi09} and~\cite[Section~11.1]{HLW06}.
See also~\cite{Babai79}.
\begin{definition}
For a finite group $G$, a function $f:G\to \CC$, and a representation $\rho:G\to \mathrm{GL}(V)$, the \textit{Fourier transform of $f$ at $\rho$} is the linear transformation $\hat{f}(\rho):V\to V$ defined by
\[\hat{f}(\rho)=\sum\limits_{g\in G} f(g)\rho(g).\]
\end{definition}
Let $\rho_{\mathrm{reg}}$ be the left regular representation of $G$ extended linearly to $\CC G$.  Then
\[\hat{f}(\rho_{\mathrm{reg}})(h)=\sum\limits_{k\in G} f(k)\rho_{\mathrm{reg}}(k)h=\sum\limits_{k\in G} f(k) kh=\sum\limits_{g\in G}f(gh^{-1})g,\]
and thus the action of $\hat{f}(\rho_{\mathrm{reg}})$ is given by the matrix $M_f$ with $(g,h)$-entry $f(gh^{-1})$, i.e., with 
\[(M_f)_{g,h}=f(gh^{-1}).\]
By the observations~(\ref{eq:a}) and~(\ref{eq:d}), the adjacency matrix $A_\Gamma$ represents the action of $\hat{\delta}_T(\rho_{\mathrm{reg}})$, and the distance matrix $D_\Gamma$ represents the action of $\hat{\ell}_T(\rho_{\mathrm{reg}})$.

Recall that $f:G\to \CC$ is a \textit{class function} if it is constant on the conjugacy classes of $G$.

\begin{lemma}
\label{le:SchurCor}
Let $\rho:G\to\mathrm{GL}(V)$ be an irreducible representation of $G$ with character $\chi_\rho$.
Let $f:G\to \CC$ be a class function.  Then $\hat{f}(\rho)=cI$ is a scalar multiple of the identity map with $c=\dfrac{1}{\chi_\rho(1)}\sum\limits_{g\in G} f(g)\chi_\rho(g)$.
\end{lemma}
\begin{proof}
This is a corollary of Schur's Lemma.
Note $\hat{f}(\rho)$ commutes with the action of $W$ on $V$, as
\begin{align*}
\rho(h)^{-1}\hat{f}(\rho)\rho(h)&=\rho(h^{-1})\sum\limits_{k\in G}f(k)\rho(k)\rho(h) \\
&=\sum\limits_{k\in G}f(k)\rho(h^{-1}kh) \\
&=\sum\limits_{g\in G}f(hgh^{-1})\rho(g) \\
&=\sum\limits_{g\in G}f(g)\rho(g)=\hat{f}(\rho).
\end{align*}
Hence $\hat{f}(\rho)$ satisfies the hypothesis of Schur's Lemma so $\hat{f}(\rho)=cI$ and taking traces yields that $c \cdot \chi_\rho(1)=\sum\limits_{g\in G} f(g)\chi_\rho(g)$.
\end{proof}

Combining Lemma~\ref{le:SchurCor} with the well-known decomposition of the regular representation results in a character formula for spectra (see~\cite{Babai79,Diaconis-Shahshahani81}).

\begin{theorem}
\label{th:charformula}
Let $G$ be a finite group and let $f:G\to \CC$ be a class function.
Then the eigenvalues of the matrix $M_f=(f(gh^{-1}))_{g,h\in G}$ are
\[\theta_\chi=\dfrac{1}{\chi(1)}\sum\limits_{g\in G} f(g)\chi(g), \hbox{ with multiplicity } \chi(1)^2,\]
where $\chi$ ranges over the irreducible characters of $G$.
\end{theorem}
\begin{proof}
The matrix $M_f=(f(gh^{-1}))$ provides the action of $\hat{f}(\rho_{\rm{reg}})$.
The left regular representation of a finite group $G$ is the direct sum
\[\rho_{\rm{reg}}=\bigoplus_{i=1}^k \chi_i(1)\rho_i\]
where $\rho_1,\dots,\rho_k$ is a list of distinct irreducible representations of $G$ with characters $\chi_1,\dots,\chi_k$. 
Since $f$ is a class function, apply Lemma~\ref{le:SchurCor} to the $i$-th summand of the decomposition of $\rho_{\rm{reg}}$ to
conclude both the formula for $\theta_{\chi_i}$ and the multiplicity since the eigenvalue $\theta_{\chi_i}$ occurs $\chi_i(1)$ times for each of the $\chi_i(1)$ copies of $\rho_i$.
\end{proof}


\headingstyle{Specific Spectra}\\
Let $W$ be a complex reflection group.  
Since the set $T$ of all reflections is closed under conjugation, the indicator function $\delta_T$ is a class function, and hence Theorem~\ref{th:charformula} implies the eigenvalues of the adjacency matrix of the Cayley graph $\Gamma(W,T)$ are 
\begin{equation}
\label{eq:adjspec}
\mu_\chi:=\dfrac{1}{\chi(1)}\sum\limits_{w\in W} \delta_T(w)\chi(w), \hbox{ with multiplicity } \chi(1)^2.
\end{equation}

The reflection length function $\ell_T$ is also a class function since
$T$ is clearly invariant under conjugation (see for example the second paragraph of Section~2.5 in~\cite{Armstrong09}).
Theorem~\ref{th:charformula} implies the eigenvalues of the distance matrix for the Cayley graph $\Gamma(W,T)$ are 
\begin{equation}
\label{eq:distspec}
\eta_\chi:=\dfrac{1}{\chi(1)}\sum\limits_{w\in W} \ell_T(w)\chi(w), \hbox{ with multiplicity } \chi(1)^2.
\end{equation}

Though not immediately clear from Equation~(\ref{eq:distspec}), Renteln~\cite[Theorem~6]{Renteln11} shows $\Gamma(W,T)$ is distance integral
for $W$ a finite real reflection group.  
The proof uses properties of $\codim(w)$, meaning the codimension of the fixed point space $V^w=\{v\in V\mid wv=v\}$, and in particular that $\ell_T(w)=\codim(w)$ for all $w$ in a real reflection group. 
However, $\ell_T=\codim$ holds only for certain complex reflection groups~\cite{Foster-Greenwood13,Shepler-Witherspoon11,Shi07a,Carter72}, and when $\ell_T\neq \codim$ we can also consider the codimension matrix 
\[C_\Gamma \hbox{ with $(v,w)$-entry } C_\Gamma(v,w)=\codim(vw^{-1}).\]
The codimension matrix is more fundamentally related to the group representation rather than the graph, but the codimension function remains a class function for complex reflection groups, so Theorem~\ref{th:charformula} applies to yield that the eigenvalues of the codimension matrix are
\begin{equation}
\label{eq:codimspec}
\xi_{\chi}:=\dfrac{1}{\chi(1)}\sum\limits_{w\in W} \codim(w)\chi(w), \hbox{ with multiplicity } \chi(1)^2.
\end{equation}

\begin{example}\label{ex:dihedralspectra}
For the dihedral group $G(r,r,2)$ of order $2r$ an easy exercise using the different character tables when $r$ is even and odd yields the following adjacency and distance spectra in both cases.
Since $G(r,r,2)$ is a real reflection group the codimension spectrum is the same as the distance spectrum.

\begin{center}
\begin{tabular}{rrlrl}
$G(r,r,2)$ & $\mu_{\chi}^{}$ & $\mathrm{multiplicity}$ & $\eta_{\chi}^{}=\xi_{\chi}$ & $\mathrm{multiplicity}$\\
\hline
& $r$ & $1$ & $3r-2$ & 1 \\
& $0$ & $2r-2$ & $r-2$ & 1 \\
& $-r$ & $1$ & $-2$ & $2r-2$ \\
\end{tabular}
\end{center}
\end{example}


\headingstyle{Spectral Radius}\\
As an aside, we briefly consider a result related to size of these eigenvalues.
The \textit{spectral radius} of a matrix $M$ is the maximum modulus of its eigenvalues.
There are numerous papers obtaining bounds on the adjacency or Laplace spectral radius for graphs satisfying certain conditions, which can be used in contrapositive form to prove a given graph does not satisfy the conditions.
Some papers investigate distance spectral radius and point to connections to chemistry (see e.g.,~\cite{Zhou-Ilic10}).
In~\cite[Conjecture~12]{Renteln11} it is conjectured that $\eta_1$ is the largest distance eigenvalue of $\Gamma(W,T)$ when $W$ is a real reflection group.
More generally, one expects that when $W$ is a complex reflection group $\eta_1$, $\mu_1$, and $\xi_1$ are the largest distance, adjacency, and codimension eigenvalues respectively.
This is easily seen to be true using the following corollary of Theorem~\ref{th:charformula} on spectral radius.\footnote{We thank Yu Chen for this observation.}

\begin{corollary}
\label{cor:radius}
Suppose that $f:G\to \CC$ is a class function on a finite group.
The spectral radius of $M_f=(f(gh^{-1}))_{g,h \in G}$ is $\sum\limits_{g \in G} |f(g)|$.
When $f:G\to \RR^{\geq 0}$ the spectral radius is the eigenvalue $\theta_1$ corresponding to the trivial character, and when all spectra $\theta_\chi$ are real (for instance when $M_f$ is symmetric), then $\theta_1$ is the largest eigenvalue.
\end{corollary}
\begin{proof}
By Theorem~\ref{th:charformula}, the eigenvalues of the matrix $M_f=(f(gh^{-1}))$ are 
\[\theta_\chi=\dfrac{1}{\chi(1)}\sum\limits_{g\in G} f(g)\chi(g), \hbox{ with multiplicity } \chi(1)^2,\]
where $\chi$ ranges over the irreducible characters of $G$.
In general $\chi(g) \in \CC$, and since $G$ is finite, $\chi(g)$ is a sum of roots of unity. 
Thus for any $g\in G$ and any irreducible character $\chi$ of $G$, the modulus $|\chi(g)|$ is at most $\chi(1)$, the dimension of the irreducible character $\chi$.
Hence 
\[|\theta_\chi|=\left| \sum\limits_{g \in G} \frac{\chi(g)}{\chi(1)} f(g)\right| \leq \sum\limits_{g \in G} \frac{|\chi(g)|}{\chi(1)} |f(g)| \leq \sum\limits_{g \in G}|f(g)|.\]
When the values of $f$ are real and nonnegative, $\sum\limits_{g \in G}|f(g)|=\sum\limits_{g \in G}f(g)=\theta_1$, and when the eigenvalues $\theta_\chi$ are all real, this means $\theta_1 \geq \theta_\chi$ for all $\chi$. 
In fact, by the Perron-Frobenius Theorem, $\theta_1 > \theta_\chi$ for all nontrivial $\chi$.
\end{proof}

\begin{corollary}
\label{cor:largesteval}
For $W$ a complex reflection group with Cayley graph $\Gamma(W,T)$, the largest adjacency, distance, and codimension eigenvalues are $\mu_1$, $\eta_1$, and $\xi_1$, respectively.
\end{corollary}

In particular, when $W$ is a real reflection group, there is a formula for the largest distance eigenvalue $\eta_1$~\cite[Corollary~11]{Renteln11} in terms of the degrees $d_1,\dots,d_n$ of a set of algebraically independent generators for the ring of polynomials invariant under the reflection group~\cite{Humphreys90},
\[\eta_1=|W|\sum\limits_{i=1}^n \dfrac{d_i-1}{d_i}=\sum\limits_{w \in W} \ell_T(w).\]
The known values of the degrees can be used to compute the first sum for types $A_n$, $B_n$, $D_n$, and $I_2(n)$ as a function of $n$, providing a value for the second sum.
The sequence for type $A_n$ appeared in the Online Encyclopedia of Integer Sequences (\texttt{www.oeis.org}) as the total number of transpositions used to write all permutations of $n+1$ letters.
The sequences for types $B_n$ and $D_n$ as the sum of the reflection lengths of all elements in each group were added during this project.


\headingstyle{Rationality of Character Sums}\\
We now review prior results related to rationality of character sums and integrality of spectra.
A useful perspective is to focus on rational conjugacy classes.
We denote by $\mathcal{C}_g$ the conjugacy class of $g$.

\begin{definition} 
Elements $g$ and $h$ of a finite group $G$ are \textit{rationally conjugate} if and only if
the cyclic subgroups $\langle g \rangle$ and $\langle h \rangle$ are conjugate.
This relation partitions $G$ into rational conjugacy classes, where
the \textit{rational conjugacy class} of $g$ is the union
$$\ds\mathcal{K}_g
=\bigcup_{\langle h \rangle = \langle g \rangle} \mathcal{C}_h
=\bigcup_{\mathrm{gcd}(d,o(g))=1}\mathcal{C}_{g^d}$$
of ordinary conjugacy classes, where $o(g)$ is the order of $g$.
\end{definition}

\begin{lemma}\cite[Lemma~3]{Renteln11}
\label{le:rational}
For any character $\chi$ of any finite group $G$ and any $g \in G$, the sum $\sum\limits_{h\in \mathcal{K}_g} \chi(h)$ is rational.
\end{lemma}
\begin{proof}
Let $g \in G$ have order $k$, let $U_k$ be the group of units modulo $k$, and let $\zeta_k$ be a primitive $k$-th root of unity.
Suppose $\rho:G \to \mathrm{GL}(V)$ is the representation of $G$ with character $\chi$ and the eigenvalues for $\rho(g)$ acting on $V$ are $\lambda_1,\dots,\lambda_n$. 
Since $\mathcal{K}_g$ can be partitioned into sets of the form 
$\{xg^s x^{-1} \mid s \in U_k \}$ where $x\in G$,
it suffices to show that $\sum\limits_{s \in U_k} \chi(g^s)$ is rational in order to conclude that $\sum\limits_{h\in \mathcal{K}_g} \chi(h)$ is rational.
Let $\sigma$ be an element of the Galois group of the cyclotomic extension $\QQ(\zeta_k)$ with $\sigma(\zeta_k)=\zeta_k^d$ for $d$ relatively prime to $k$.
For each eigenvalue $\lambda_j$, 
\[\sum\limits_{s\in U_k} \lambda_j^s=\sum\limits_{s\in U_k} \lambda_j^{ds}=\sigma\left(\sum\limits_{s\in U_k} \lambda_j^s\right),\]
and thus $\sum\limits_{s\in U_k} \lambda_j^s$ is rational.
Therefore 
\[\sum\limits_{s\in U_k} \chi(g^s)=\sum\limits_{s\in U_k}\sum\limits_{j=1}^n\lambda_j^s=\sum\limits_{j=1}^n\sum\limits_{s\in U_k}\lambda_j^s\] 
is also rational and it follows that $\sum\limits_{h\in \mathcal{K}_g} \chi(h)$ is rational.
\end{proof}

\begin{lemma}\cite[Lemma~4]{Renteln11}
\label{le:prodrational}
If $\chi$ is an irreducible character of a finite group $G$ and $f:G \to\ZZ$ is constant on rational conjugacy classes, then $\displaystyle\sum_{g\in G} f(g)\chi(g)$ is rational.
\end{lemma}
\begin{proof}
Write $G$ as a disjoint union of rational conjugacy classes $\mathcal{K}_{g_i}$ for some choice of representatives $g_1,\dots,g_l \in G$.  
Then by Lemma~\ref{le:rational}
\[\sum_{g\in G} f(g)\chi(g)=\sum\limits_{i=1}^l\sum_{h\in \mathcal{K}_{g_i}} f(h)\chi(h)=\sum\limits_{i=1}^l f(g_i)\sum_{h\in \mathcal{K}_{g_i}} \chi(h)\in\QQ.\]
\end{proof}

\begin{theorem}
\label{th:suffcondforintegrality}
Let $G$ be any finite group.  If $f:G \to\ZZ$ is integer-valued and 
constant on rational conjugacy classes, then the matrix $M_f=(f(gh^{-1}))_{g,h\in G}$
has integral eigenvalues.
\end{theorem}
\begin{proof}
  By Theorem~\ref{th:charformula}, the eigenvalues of the matrix $M_f=(f(gh^{-1}))$
  are
  $$\theta_\chi=\dfrac{1}{\chi(1)}\sum\limits_{g\in G} f(g)\chi(g), 
  \hbox{ with multiplicity } \chi(1)^2,$$
  where $\chi$ ranges over the irreducible characters of $G$.
  Since $f$ is constant on rational conjugacy classes,
  we may apply Lemma~\ref{le:prodrational} to conclude
  all eigenvalues are rational.
  Since $f$ is integer-valued, the characteristic polynomial of $M_f$
  is monic with integer coefficients, and hence its roots are in fact integers.  
\end{proof}
     
\begin{corollary}
\label{cor:adjcodimintegrality}
Let $W$ be an irreducible complex reflection group, and let $T$ be the set of all reflections in $W$.  
The Cayley graph $\Gamma(W,T)$ is adjacency integral and the codimension spectrum of $W$ is integral.
\end{corollary}
\begin{proof}
Since a nontrivial power of a reflection is again a reflection, the indicator function $\delta_T$ is constant on rational conjugacy classes.
The codimension function is also constant on rational conjugacy classes because the number of eigenvalues equal to $1$ is unchanged when an element is raised to a power relatively prime to its order\textemdash see~\cite[Lemma~2]{Renteln11} and note it does not depend on $W$ being real.
Hence the adjacency matrix $(\delta_T(vw^{-1}))_{v,w\in W}$ and codimension matrix $(\codim(vw^{-1}))_{v,w\in W}$ have integral eigenvalues by Theorem~\ref{th:suffcondforintegrality}.  
\end{proof}

\noindent
As an illustration of integrality, Table~\ref{ta:distspecGr12Gr13} gives
formulas in terms of $r$ for the codimension, equivalently distance,
spectra for the groups $G(r,1,2)$ and $G(r,1,3)$ found using 
Theorem~\ref{th:codimspec}.

\section{Integrality of the distance spectra}
\label{sec:integrality}

To conclude the Cayley graphs $\Gamma(W,T)$ are distance integral it suffices to show that 
the length function $\ell_T$ is constant on rational conjugacy classes. 
In~\cite{Renteln11} this is done for $W$ a real reflection group by using that $\ell_T(w)=\codim(w)$ for all $w$ in such a group.
However, $\ell_T=\codim$ if and only if $W$ is $G(r,1,n)$ or a real reflection group~\cite{Carter72,Shi07a,Shepler-Witherspoon11,Foster-Greenwood13}.
In Lemma~\ref{le:TlengthconstKrimprim} and Lemma~\ref{le:TlengthconstKrprim} we show that $\ell_T$ is in fact constant on rational conjugacy classes
even when $\ell_T\neq \codim$.  For the groups $G(r,p,n)$, the proof
uses reflection-preserving automorphisms, while for the exceptional 
reflection groups we rely on computer calculations using GAP3~\cite{gap3}.

Note that for a reflection group $W$, an automorphism
that permutes the set of reflections will
preserve reflection length.
From this point of view, the length function $\ell_T$ is constant on conjugacy classes because
the inner automorphisms permute the set of reflections.
To prove the length function is constant on rational conjugacy classes
of $G(r,p,n)$, we will show that 
for any pair of rationally conjugate elements $g$ and $h$, 
there is a reflection-preserving automorphism of $G(r,p,n)$ that maps $g$ to $h$.

To begin, identify the group $G(r,1,n)$ 
with its image under the standard monomial representation so that
each element is represented by a monomial matrix with entries
in $\QQ(\zeta_r)$, where $\zeta_r=e^{2\pi i/r}$.
For a field automorphism $\zeta_r\mapsto\zeta_r^x$ of $\QQ(\zeta_r)$ and
an element $g$ in $G(r,1,n)$,
let $\alpha_x(g)$ be the result of applying the field automorphism to
the entries of the matrix of $g$.
As observed in \cite[Lemma 3.5]{Shi-Wang09},
the map $\alpha_x$ is a reflection-preserving
automorphism of $G(r,1,n)$ and also of its subgroups $G(r,p,n)$.  
We call $\alpha_x$ a {\it Galois automorphism}
since it arises from an element of the Galois group of $\QQ(\zeta_r)$
and is an example of a Galois automorphism in the 
sense described in~\cite{Marin-Michel10}.

\begin{lemma}\label{le:auto}
  For $g$ in $G(r,1,n)$ and $d$ relatively prime to
  the order of $g$, there exists a Galois automorphism
  of $G(r,1,n)$ that maps $g$ to a conjugate of $g^d$.
\end{lemma}

\begin{proof}
  Let $g$ be an element of $G(r,1,n)$, and let $d\neq1$ be relatively prime to the order of $g$.  Let $\alpha_x$ denote a Galois automorphism of $G(r,1,n)$.
  Recall that elements of $G(r,1,n)$ are conjugate if and only if 
  they have the same multiset of (cycle-size, cycle-sum) pairs.  If the pairs
  for $g$ are $(k_1,c_1),\ldots,(k_m,c_m)$, then the pairs for
  $g^d$ are $(k_i,dc_i)$, while those for $\alpha_x(g)$ are $(k_i,xc_i)$.
  
  If $d$ is relatively prime to $r$, then $\zeta_r\mapsto\zeta_r^d$ is
  a field automorphism of $\QQ(\zeta_r)$, and $\alpha_d$ is a 
  Galois
  automorphism of $G(r,1,n)$ that maps $g$ to a conjugate of $g^d$, so we are done.  
  However, if $d$ is not relatively prime to $r$, then,  
  recalling that the cycle-sums are only well-defined modulo $r$,
  we seek to define $\alpha_x$ using a power $x$ relatively prime to $r$ such that
  $xc_i\equiv dc_i\bmod r$ for each $i$.
    Note that, for $1\leq i\leq m$,
    $$            
    xc_i\equiv dc_i\bmod r 
    \phantom{wooo}\Leftrightarrow\phantom{wooo} 
    (\zeta_r^{c_i})^{x}=(\zeta_r^{c_i})^d
    \phantom{wooo}\Leftrightarrow\phantom{wooo} 
    x\equiv d\bmod r_i,
    $$
  where $r_i=\order(\zeta_r^{c_i})$.   
  Thus, it suffices to find a power $x$ relatively prime to $r$
  with $x\equiv d\bmod \ds\lcm(r_1,\ldots,r_m)$.
  The key observation is that $x$ is relatively prime to $r$ if and
  only if $x\not\equiv0\bmod q$ for all primes $q$ dividing $r$.
  With this in mind, let $x$ be a solution\textemdash guaranteed to exist by the Chinese Remainder Theorem\textemdash to the system of congruences
  $$
  \left\{\begin{array}{ll}
         x\equiv 1\bmod q & \text{ for all primes $q$ dividing $r$ but not $\lcm(r_1,\ldots,r_m)$}\\
         x\equiv d\bmod \lcm(r_1,\ldots,r_m) &
         \end{array}\right..
  $$
  If $q$ is a prime dividing $r$ but not $\lcm(r_1,\ldots,r_m)$, then $x\not\equiv0\bmod q$ by construction.
  If $q$ is a prime dividing $\lcm(r_1,\ldots,r_m)$, then $x\equiv d\bmod q$.
  But $q$ divides $\order(g)=\lcm(k_1,\ldots,k_m,r_1,\ldots,r_m)$, and $d$ is relatively prime to $\order(g)$, 
  so $d\not\equiv0\bmod q$. 
\end{proof}

We now map an element of $G(r,p,n)$ to any rationally conjugate element
by compositions of Galois automorphisms and inner automorphisms of $G(r,1,n)$
that preserve the normal subgroup $G(r,p,n)$.

\begin{lemma} 
\label{le:TlengthconstKrimprim}
For the complex reflection groups $G(r,p,n)$, absolute reflection length
$\ell_T$ is constant on rational conjugacy classes.
\end{lemma}

\begin{proof}
  Let $g$ and $h$ be rationally conjugate elements of $G(r,p,n)$, so
  $h$ is a $G(r,p,n)$-conjugate of $g^d$ for some $d$ relatively prime
  to the order of $g$.
  By Lemma~\ref{le:auto}, there exists a Galois automorphism $\alpha_x$
  of $G(r,1,n)$ such that 
  $\alpha_x(g)$ is a $G(r,1,n)$-conjugate of $g^d$, and hence also of $h$.  Let 
  $\beta$ be an inner automorphism of $G(r,1,n)$ such that
  $\beta(\alpha_x(g))=h$.
  Since $G(r,p,n)$ is a normal subgroup of $G(r,1,n)$, the map $\beta$ 
  preserves $G(r,p,n)$.
  Now $\beta\circ\alpha_x$ is a reflection-preserving automorphism of
  $G(r,p,n)$ that maps $g$ to $h$, and so $g$ and $h$ must have the same reflection length.
\end{proof}

Analysis of reflection-preserving automorphisms is more subtle for the
exceptional complex reflection groups, and that in~\cite{Marin-Michel10}
appears to involve computer calculations.  So for the groups $G_4$-$G_{37}$, we instead compute absolute reflection length and rational conjugacy classes
using code\footnote{Code is posted in the repository \url{http://github.com/fostergreenwood/spectra}.}\addtocounter{footnote}{-1}\addtocounter{Hfootnote}{-1} for {\sc GAP3}~\cite{gap3} and observe:

\begin{lemma} 
\label{le:TlengthconstKrprim}
For the irreducible exceptional complex reflection groups, absolute reflection
length $\ell_T$ is constant on rational conjugacy classes.
\end{lemma} 

Our main result now follows.

\begin{theorem} Let $W$ be an irreducible complex reflection group, and let $T$
	be the set of all reflections in $W$.  
	Then the Cayley graph $\Gamma(W,T)$ is distance integral.
	\label{th:distintegrality}
\end{theorem}
\begin{proof}
	By Theorem~\ref{th:suffcondforintegrality}, this is an immediate corollary of Lemma~\ref{le:TlengthconstKrimprim} and
	Lemma~\ref{le:TlengthconstKrprim}.
	For the exceptional reflection groups, we additionally verified integrality
	by inspection of eigenvalues calculated\footnotemark\ 
	 in GAP3~\cite{gap3} using Equation~(\ref{eq:distspec}).	
\end{proof}

\noindent	
Note that if Lemma~\ref{le:TlengthconstKrprim} could be proven by other than computational means, it would complete a purely non-computational proof of Theorem~\ref{th:distintegrality}.


\section{Combinatorial formula for codimension spectra}
\label{sec:combformula}

By Corollary~\ref{cor:adjcodimintegrality}, the codimension spectrum of a complex reflection group $W$
is integral.
In this section, we provide a combinatorial formula that can be used to 
compute the codimension spectra for the reflection groups $G(r,1,n)$ without using character values, 
thus supplying an alternative approach to their integrality.
This will involve reviewing the use of symmetric functions to describe roots of Poincar\'{e} polynomials.
\\

\headingstyle{Poincar\'{e} Polynomials}\\
Let $W$ be a complex reflection group acting on $V \simeq \CC^n$ by a reflection representation. 
Given an irreducible character $\chi$ of $W$, define the Poincar\'{e} polynomials
\[R_\chi(t)=\frac{1}{\chi(1)}\sum\limits_{w\in W} \chi(w)t^{\codim(V^w)}
\qquad \hbox{ and } \qquad
R_\chi^*(t)=\frac{1}{\chi(1)}\sum\limits_{w\in W} \chi(w)t^{\dim(V^w)}, \]
where $V^w$ is the fixed point space of $w$.
Since $\codim V^w=n-\dim V^w$, 
the Poincar\'{e} polynomials $R_{\chi}$ and $R_{\chi}^*$ are
reciprocal polynomials, i.e., satisfy the relations
$$
R_{\chi}^*(t)=t^n R_{\chi}(t^{-1}) \qquad\text{ and }\qquad R_{\chi}(t)=t^n R_{\chi}^*(t^{-1}).
$$
Note that $(t-c)$ is a factor of $R_{\chi}^*(t)$ if and only if $t(t^{-1}-c)=(1-ct)$ is a factor of $R_{\chi}(t)$.
Also, the identity is the unique element with $\dim V^w=n$, so $R_{\chi}^*(t)$
is monic. 
Referring to Equation~\ref{eq:codimspec}, note that the eigenvalues of the codimension matrix can be found as
\begin{equation}
\label{eq:derivative}
\xi_{\chi}=\left.\frac{dR_\chi}{dt}\right|_{t=1}.
\end{equation}
For the remainder of this section, we focus on the Poincar\'{e} polynomials for
the reflection group $G(r,1,n)$ acting on the vector space $V\cong\CC^n$ by its 
standard monomial representation.  
The characters of $G(r,1,n)$ are indexed by combinatorial objects $\lambda$
(partitions for the symmetric group $S_n=G(1,1,n)$ and $r$-tuples of partitions
in general), so we use the notation $R_{\lambda}$ for the corresponding Poincar\'{e} polynomial.

\begin{example} 
\label{ex:trivchar}
We begin with a classical result in the theory
of reflection groups~\cite{Solomon63}.  If $W$ is a reflection group with basic
invariant polynomials of degrees $d_1,\ldots,d_n$, then  
the Poincar\'{e} polynomial $R_{\text{triv}}(t)$ factors using the 
exponents $m_i=d_i-1$ of the group:
$$
R_{\text{triv}}(t)=\sum_{w\in W}t^{\codim(V^w)}=\prod_{i=1}^{n}(1+m_it).
$$
In particular, when the symmetric group $S_n$ acts on $V\cong\CC^n$ by its permutation representation, the basic invariant polynomials are the elementary symmetric functions
in $n$ variables and have degrees $1,2,\ldots,n$, so
$$
R_{\text{triv}}(t)=\sum_{\sigma\in S_n}t^{\codim(V^{\sigma})}=(1+t)(1+2t)\cdots(1+(n-1)t)
$$
and
$$
R_{\text{triv}}^*(t)=t^nR_{\text{triv}}(t^{-1})=t(t+1)(t+2)\cdots(t+(n-1)).
$$
The roots of $R_{\text{triv}}^*(t)$ are the $-c_{ij}$ for the Young diagram
$$
\begin{array}{|c|c|c|c|c|}
\hline
 0 & 1 & 2 & \cdots & n-1\\
 \hline
\end{array}
$$
associated to the trivial representation, 
where $c_{ij}=j-i$ is the \textit{content} of the box in row $i$ and column $j$ of a Young diagram, recording which diagonal the box is on.
\end{example}

There is a generalization of Example~\ref{ex:trivchar} that factors the Poincar\'{e} polynomial $R_{\lambda}^*(t)$ corresponding to any irreducible character of the symmetric group. 

\begin{theorem}[Poincar\'{e} Polynomials for $S_n$]\label{th:factorPoincarepoly}
If $\chi_{_{\lambda}}$ is the character of the symmetric group $S_n$
corresponding to the partition $\lambda$,
then the Poincar\'{e} polynomial $R_{\lambda}^*(t)$ factors as
$$
R_{\lambda}^*(t)=\prod_{ij\in\lambda}(t+c_{ij}),
$$
where $c_{ij}=j-i$ is the content of the $(i,j)$-box in the Young diagram for $\lambda$.
\end{theorem}

This theorem is~\cite[exercise 7.50]{Stanley99} and the solution utilizes~\cite[Theorem 7.21.2]{Stanley99}, which is noted to have first been explicitly stated by Stanley in~\cite{Stanley71}.
The theorem also appeared in~\cite{Molchanov82}, which provides a similar factorization for the $R_{\lambda \mu}^*(t)$ for the real reflection group $\{\pm 1\}^n \rtimes S_n$ of type $B_n$, a sum formula for the analogous $R_{\chi}^*(t)$ for the real reflection group of type $D_n$, and a factorization result for the $R_{\chi}^*(t)$ for dihedral groups.
Renteln~\cite{Renteln11} uses Theorem~\ref{th:factorPoincarepoly} and Equation~(\ref{eq:derivative}) to provide a combinatorial formula for the codimension (equivalently distance) spectra in type $A$, leaves similar formulas in types $B$ and $D$ to the reader, and includes the spectra for dihedral type. 
 
We extend these results by showing the Poincar\'{e} polynomials for the groups $G(r,1,n)$ also factor using
contents of Young diagrams.  Example~\ref{ex:codimGr12Gr13} illustrates the theorem for the groups $G(r,1,2)$ and $G(r,1,3)$.

\begin{theorem}[Poincar\'{e} Polynomials for $G(r,1,n)$]
\label{th:codimspec}
If $\chi_{_{\lambda}}$ is the character of $G(r,1,n)$ indexed by the
$r$-tuple $\lambda$ of partitions $\lambda(0),\ldots,\lambda(r-1)$,
then the Poincar\'{e} polynomial $R_{\lambda}^*(t)$ factors as
$$
R_{\lambda}^*(t)=\prod_{ij\in\lambda(0)}(t+r-1+rc_{ij})\prod_{ij\in\lambda(1)}(t-1+rc_{ij})\cdots\prod_{ij\in\lambda(r-1)}(t-1+rc_{ij}),
$$
where $c_{ij}=j-i$ is the content of the $(i,j)$-box in a Young diagram.
\end{theorem}

\headingstyle{Symmetric Functions and Identities}\\
Before proving Theorem~\ref{th:codimspec}, we review the cast of symmetric functions in countably many variables
$\{x_i\}_{i=1}^{\infty}$ and relevant identities.

The \textit{power sum symmetric functions} $p_d$ are defined by the generating function
$$
P(t)=p_1+p_2t+p_3t^2+\cdots = \sum_{i\geq1}\frac{x_i}{1-x_it},
$$
so $p_d$ is the sum of all monomials of the form $x_i^d$.
The \textit{complete homogeneous symmetric functions} $h_d$ are defined by the generating function
$$
H(t)=h_0+h_1t+h_2t^2+\cdots = \prod_{i\geq1}\frac{1}{1-x_it},
$$
so $h_d$ is the sum of all monomials in the variables $x_i$ with total degree $d$.
 
Recall that the conjugacy class of an element in the symmetric group is
determined by the orders of the cycles in the element's disjoint cycle
decomposition.  Conjugacy class types can be catalogued by partitions, which can be given as Young diagrams, and 
also by monomials in the symmetric functions $p_{_{d}}$.  
The \textit{conjugacy class indicator monomial} $p_{_{\sigma}}$ records the conjugacy class type of a permutation $\sigma$ as a product of power sums $p_{_d}$, where the multiplicity of $p_{_{d}}$ is the number of cycles of order $d$ in the disjoint cycle decomposition of $\sigma$.

\begin{example}
\label{ex:key}
The following table shows the Young diagrams and indicator monomials for the five conjugacy classes of 
the symmetric group $S_4$.
$$
{\renewcommand{\arraystretch}{1.5}
\begin{array}{cp{1cm}ccccc}
S_4 & &(1)(2)(3)(4) & (12)(3)(4) & (12)(34) & (123)(4) & (1234) \\
\hline
\rm{diagram} & & \sub\yng(1,1,1,1) & \sub\yng(2,1,1) & \sub\yng(2,2) & \sub\yng(3,1) & \sub\yng(4) \\
\rm{monomial}   &   & p_1^4           & p_2p_1^2      & p_2^2     & p_3p_1    & p_4 \\
\hline
\end{array}}
$$
\end{example}

The \textit{Schur function} $s_{\lambda}$ associated to a partition $\lambda$ is 
defined using determinants of alternant matrices in the variables $x_i$;
however, it is the following identities (see~\cite[p.~41,~114]{Macdonald98}),
expressing $s_{\lambda}$ in terms of the power sum and
complete homogeneous symmetric functions, that are relevant for our purposes:
\begin{equation}
\label{eq:complete}
s_{_{\lambda}}=\det(h_{_{\lambda_i-i+j}}) 
\end{equation}
\begin{equation}
\label{eq:averaging}
s_{_{\lambda}}=\frac{1}{n!}\sum_{\sigma\in S_n}\chi_{_{\lambda}} (\sigma)p_{_{\sigma}}. 
\end{equation}
As a special case (see~\cite[p.~25]{Macdonald98}), if $\lambda$ is the partition consisting of
a single row of $d$ boxes, then the character $\chi_{_{\lambda}}$ is the trivial character 
for the symmetric group on $d$ letters, and
$$
h_d=s_d=\frac{1}{d!}\sum_{\sigma\in S_d}p_{\sigma}.
$$
Thus the entries of the matrix $(h_{_{\lambda_i-i+j}})$ may also be viewed as functions of the power sums $p_d$.

\begin{remark}
It is typical to work with the symmetric functions in an infinite number
of variables so that the power sum symmetric functions $p_{d}$ are algebraically independent over $\QQ$.
See the discussion in~\cite[Section~I.2]{Macdonald98}.
Consequently we can view $p_1,p_2,\ldots$ as indeterminates.
Further, any identity that can be expressed in terms of the power sum symmetric functions
also holds for any set of countably many indeterminates over $\QQ$.  
\end{remark}

\headingstyle{Substitution}\\
Since the dimension of the fixed point space of an element
of the symmetric group is the number of cycles in the element's disjoint cycle decomposition, the substitution
$$
p_{_d}=t\quad\text{for each $d\geq1$}
$$ 
transforms the conjugacy class indicator monomials into $p_{_{\sigma}}(t)=t^{\dim V^{\sigma}}$,
and, in turn,
$$
\det(h_{_{\lambda_i-i+j}}(t))
=s_{_{\lambda}}(t)=
\frac{1}{n!}\sum_{\sigma\in S_n}\chi_{_{\lambda}}(\sigma)p_{_{\sigma}}(t).
$$
Identity~(\ref{eq:averaging}) establishes that the Poincar\'{e}
polynomial $R_{\lambda}^*(t)$ has the same roots as the Schur function $s_{_{\lambda}}(t)$.
Indeed, after substitution, the Schur function is
$$
s_{_{\lambda}}(t)
=\frac{1}{n!}\sum_{\sigma\in S_n}\chi_{_{\lambda}}(\sigma)t^{\dim V^{\sigma}}
=\frac{\chi_{_{\lambda}}(1)}{n!}R_{\lambda}^*(t).
$$
Thus, the Schur function $s_{_{\lambda}}(t)$ and the Poincar\'{e} polynomial $R_{\lambda}^*(t)$
are the same up to multiplication by a nonzero scalar:
\begin{equation}\label{eq:sameroots}
s_{_{\lambda}}(t)\doteq R_{\lambda}^*(t).
\end{equation}
It is then possible to use Identity~(\ref{eq:complete}), along with
properties of determinants and the factorization of $R_{\text{triv}}^*(t)$ from Example~\ref{ex:trivchar},
to determine the roots of $s_{_{\lambda}}(t)$ and $R_{\lambda}^*(t)$ (see~\cite{Molchanov82}). 
\\

\headingstyle{More Symmetric Functions and Identities}\\
For the groups $G(r,1,n)$, we will show that the roots of the Poincar\'{e} polynomials 
coincide with the roots of generalized Schur functions defined using multiple sets of 
symmetric functions.

Let $C_r$ be the cyclic group of order $r$ generated by $\zeta_r=e^{2\pi i/r}$, and let
$\widehat{C_r}$ be the set of irreducible characters $\gamma_0,\ldots,\gamma_{r-1}$ of $C_r$, where $\gamma_k(\zeta_r)=\zeta_r^k$.
Let $\{p_{d,c}\mid d\geq1, c\in C_r\}$ be a set of algebraically independent indeterminates over $\CC$.
For each $c$ in $C_r$, the set $\{p_{d,c}\mid d\geq1\}$ may be viewed as a set of
power sum symmetric functions (in a sequence of underlying variables $x_{i,c}$).
In turn, we have corresponding complete homogeneous symmetric functions $h_{_{d,c}}$
and Schur functions $s_{_{\lambda,c}}$.

For each degree $d\geq 1$ and each character $\gamma_{k}$ in $\widehat{C_r}$, make a change of
variables and define
$$
p_{_{d,k}}=\frac{1}{r}\sum_{c\in C_r}\gamma_k(c)p_{_{d,c}}.
$$
Thus if $X$ is the character table of $C_r$, then, for each degree $d$, the change of basis matrix from
$\{p_{_{d,c}}:c\in C_r\}$ to $\{p_{_{d,k}}:\gamma_k\in\widehat{C_r}\}$ is $\frac1r X$.
For each $\gamma_k$, the elements $p_{_{d,k}}$ are algebraically independent
and may thus be viewed as another set of power sum symmetric functions 
(in another set of underlying variables $x_{i,k}$).  In turn, we have
the related sets of complete homogeneous symmetric functions $h_{_{d,k}}$
and Schur functions $s_{_{\lambda,k}}$.


The conjugacy class types in the group $G(r,1,n)$ are catalogued by $r$-tuples of partitions and also by monomials
in the power sums $p_{_{d,c}}$.  
The \textit{conjugacy class indicator monomial}
$P_{\rho}$ records the conjugacy class type $\rho$ as a product of
indeterminates $p_{_{d,c}}$, where the multiplicity of $p_{_{d,c}}$ is the number
of blocks $\delta_i\sigma_i$ with cycle-size $d$ and cycle-product $c$.

\begin{example} 
\label{ex:Prho}  
Conjugacy class types in the group $G(3,1,14)$ are indexed by triples of partitions whose
sizes add up to $14$.  For instance, let $\zeta=e^{2\pi i/3}$, and
let $\rho=(\rho(1),\rho(\zeta),\rho(\zeta^2))$ be the triple of partitions
$$
\rho(1)=\yng(3,2), \phantom{woo}
\rho(\zeta)=\yng(2,1,1,1), \phantom{woo}
\rho(\zeta^2)=\yng(4).
$$
The lengths of the rows in the partition $\rho(c)$ correspond to the cycle-sizes of the blocks $\delta_i\sigma_i$
with cycle-product $c$.
Thus, the elements in the conjugacy class of type $\rho$ have
\begin{itemize}
 \item a $3$-cycle and $2$-cycle with cycle-product $1$,
 \item a $2$-cycle and three $1$-cycles with cycle-product $\zeta$, and
 \item a $4$-cycle with cycle-product $\zeta^2$.
\end{itemize}
The corresponding conjugacy class indicator monomial is
$
P_{\rho}=p_{3,1}^{}p_{2,1}^{} p_{2,\zeta}^{}p_{1,\zeta}^{3} p_{4,\zeta^2}^{}.
$
\end{example}

Finally, each character of the group $G(r,1,n)$ is indexed by 
an $r$-tuple $\lambda$ of partitions $\lambda(0),\ldots,\lambda(r-1)$ whose
sizes add up to $n$.
The generalized Schur function corresponding to $\lambda$ is the polynomial
$$
S_{\lambda}=\prod_{\gamma_k\in\widehat{C_r}} s_{_{\lambda(k),k}}.
$$
Parallel to Identity~(\ref{eq:averaging}), the following identity relates the generalized Schur function and
the conjugacy class indicator monomials for $G=G(r,1,n)$ (see~\cite[Appendix B.9]{Macdonald98}):
$$
S_{\lambda}=
\dfrac{1}{|G|}\sum_{g\in G}\chi_{_{\lambda}}(g)P_{_g}.
$$
We now prove Theorem~\ref{th:codimspec} for the groups $G(r,1,n)$
using a substitution analogous to that discussed previously in the case of the symmetric group. 

\begin{proof}[Proof of Theorem~\ref{th:codimspec}]
The dimension of the fixed point space of an element
of $G=G(r,1,n)$ is the number of cycles with cycle-product one, so the
substitution
$$
p_{_{d,c}}=\left\{\begin{array}{ll}
t & \text{if $c=1$} \\
1 & \text{else}
\end{array}\right.
$$
transforms the conjugacy class indicator monomials into $P_{_{g}}(t)=t^{\dim V^g}$,
and, in turn,
$$
S_{\lambda}(t)=\frac{1}{|G|}\sum_{g\in G}\chi_{_{\lambda}}(g)t^{\dim V^g}
=\frac{\chi_{_{\lambda}}(1)}{|G|}R_{\lambda}^*(t).
$$
Thus the generalized Schur polynomial and the Poincar\'{e} polynomial have
the same roots.  

We examine the effect of the substitution on the
generalized Schur polynomial.
Since the character table of the cyclic group relates the 
variables $p_{_{d,c}}$ and $p_{_{d,k}}$, we have, after substitution,
$$
p_{_{d,0}}(t)=\frac{t+r-1}{r},\quad
h_{_{d,0}}(t)=h_{_d}\left(\frac{t+r-1}{r}\right), \quad \text{ and }
s_{_{\lambda(0),0}}(t)=s_{_{\lambda(0)}}\left(\frac{t+r-1}{r}\right).
$$
And for $k\neq0$, we have
$$
p_{_{d,k}}(t)=\frac{t-1}{r},\quad
h_{_{d,k}}(t)=h_{_d}\left(\frac{t-1}{r}\right), \quad \text{ and }
s_{_{\lambda(k),k}}(t)=s_{_{\lambda(k)}}\left(\frac{t-1}{r}\right).
$$
Then the generalized Schur polynomial becomes
$$
S_{_{\lambda}}(t)=s_{_{\lambda(0)}}\left(\frac{t+r-1}{r}\right)
s_{_{\lambda(1)}}\left(\frac{t-1}{r}\right)\cdots 
s_{_{\lambda(r-1)}}\left(\frac{t-1}{r}\right),
$$
and by Equation~\ref{eq:sameroots} (relating Schur functions and Poincar\'{e} polynomials for the symmetric group),
$$
R_{_{\lambda}}^*(t)\doteq
R_{_{\lambda(0)}}^*\left(\frac{t+r-1}{r}\right)
R_{_{\lambda(1)}}^*\left(\frac{t-1}{r}\right)\cdots
R_{_{\lambda(r-1)}}^*\left(\frac{t-1}{r}\right).
$$
Finally, we use Theorem~\ref{th:factorPoincarepoly} for the symmetric group to factor each $R_{_{\lambda(k)}}^*$.  
Note that $R_{\lambda}^*(t)$ is monic, so clearing leading coefficients 
completes the factorization formula for $R_{\lambda}^{*}(t)$.
\end{proof}

\begin{remark}When $p\neq1$, the Poincar\'{e} polynomials $R_{\chi}^*(t)$ for $G(r,p,n)$ do not
always have all integer roots.  This is perhaps not too surprising since in the case $r=2$, for example,
Molchanov~\cite{Molchanov82} expresses the Poincar\'{e} polynomials of
the groups of type $D_n$ as averages of those of the groups of type $B_n$.
We suspect consideration of the Clifford theory 
relating characters of $G(r,p,n)$ to those of $G(r,1,n)$ might allow for generally expressing
the Poincar\'{e} polynomials of $G(r,p,n)$ as averages of those
of $G(r,1,n)$.  However, we did not pursue this further.
\end{remark}


\begin{example}\label{ex:codimGr12Gr13} As an example, we compute the distance, or equivalently codimension, spectra for $G(r,1,2)$ and $G(r,1,3)$, whose characters are indexed by the $r$-tuples $\lambda=(\lambda(0),\lambda(1),\ldots,\lambda(r-1))$  of partitions whose sizes add up to $2$ and $3$ respectively. 
The form of $\lambda$ and its contents are indicated in the first two columns
of Tables~\ref{distspecG(r,1,2)},~\ref{nonzerodistspecG(r,1,3)},
and~\ref{zerodistspecG(r,1,3)}.
Based on the formula in Theorem~\ref{th:codimspec},
the ordering of the partitions $\lambda(1),\ldots,\lambda(r-1)$ does not
affect $R^*_{\lambda}(t)$, so a row in one of the tables accounts for the set
all tuples differing only by a reordering of the partitions in
slots $1,\ldots,r-1$.
Given $\lambda$, it is straightforward to compute the Poincar\'{e} polynomial $R_{\lambda}^{*}(t)$ using Theorem~\ref{th:codimspec}, shift to $R_\lambda(t)=t^nR_{\lambda}^{*}(t^{-1})$, and compute the spectrum via $\xi_{\lambda}=R'_{\lambda}(1)$.
The multiplicity of each $\xi_{\lambda}$ appearing in a row of one of the tables is given in the final column as 
\[\mathrm{Mult}(\xi_{\lambda})=\chi_{\lambda}(1)^2(\hbox{\# of possible $r$-tuples of the form $\lambda$}).\]
Totaling the multiplicities for common values of $\xi_{\lambda}$ yields the spectra for $G(r,1,2)$ and $G(r,1,3)$ shown in Table~\ref{ta:distspecGr12Gr13}.


\begin{center}
\begin{table}[h]
\caption{Distance, equivalently codimension, spectrum of $\Gamma(G(r,1,2),T)$ and $\Gamma(G(r,1,3),T)$}
\label{ta:distspecGr12Gr13}
\begin{tabular}{rrlrrl}
$G(r,1,2)$ & $\eta_{\lambda}=\xi_{\lambda}^{}$ & $\mathrm{Multiplicity}$ & $G(r,1,3)$ & $\eta_{\lambda}=\xi_{\lambda}^{}$ & $\mathrm{Multiplicity}$ \\
\hline
& $4r^2-3r$ & $1$ && $18r^3-11r^2$ & $1$ \\
& $r$ & $r-1$ && $r^2$ & $13r-12$ \\
& $0$ & $2r^2-6r+4$ && $0$ & $6r^3-33r+27$ \\
& $-r$ & $5r-4$ && $-r^2$ & $9r-9$ \\
& & && $-2r^2$ & $11r-7$
\end{tabular}
\end{table}
\end{center}

\begin{landscape}

\begin{center}
\begin{table}
\caption{Distance, equivalently codimension, spectrum of $\Gamma(G(r,1,2),T)$} 
\label{distspecG(r,1,2)}
		\[
		{\renewcommand{\arraycolsep}{12pt}\renewcommand{\arraystretch}{2.5}
			\begin{array}{llcccc}
			\lambda(0) & \lambda(i), i \geq 1 & R_{\lambda}^{*}(t) & R_{\lambda}(t) & \xi_{\lambda}^{}=R'_{\lambda}(1) & \mathrm{Mult}(\xi_{\lambda}) \\
			\hline
			\vspace{.25cm}
			\scriptsize{\raisebox{.2em}{\begin{ytableau} 0 & 1 \end{ytableau}}} & & (t+r-1)(t+2r-1) & (1+(r-1)t)(1+(2r-1)t)& 4r^2-3r & 1  \\
			\vspace{.25cm}
			& \scriptsize{\raisebox{.2em}{\begin{ytableau} 0 \\ $-1$ \end{ytableau}}} & (t-1)(t-r-1) & (1-t)(1-(r+1)t)& r & r-1 \\
			\vspace{.25cm}
			& \scriptsize{\raisebox{.2em}{\begin{ytableau} 0 \end{ytableau}}\,,\, \raisebox{.2em}{\begin{ytableau} 0 \end{ytableau}}} & (t-1)^2 & (1-t)^2 & 0 & 4{r-1 \choose 2} \\
			\vspace{.25cm}
			\scriptsize{\raisebox{.2em}{\begin{ytableau} 0 \\ $-1$ \end{ytableau}}} & & (t-1)(t+r-1) & (1-t)(1+(r-1)t) & -r & 1 \\
			\vspace{.25cm}
			\scriptsize{\raisebox{.2em}{\begin{ytableau} 0 \end{ytableau}}} & \scriptsize{\raisebox{.2em}{\begin{ytableau} 0 \end{ytableau}}} & (t-1)(t+r-1) & (1-t)(1+(r-1)t) & -r & 4(r-1) \\
			\vspace{.25cm}
			& \scriptsize{\raisebox{.2em}{\begin{ytableau} 0 & 1 \end{ytableau}}} & (t-1)(t+r-1) & (1-t)(1+(r-1)t) & -r & r-1 \\
			\hline
			\end{array}}
		\]
		\end{table}
\end{center}
		
\begin{center}
\begin{table}
\caption{Nonzero distance, equivalently codimension, spectrum of $\Gamma(G(r,1,3),T)$} 
\label{nonzerodistspecG(r,1,3)}
				\[
		{\renewcommand{\arraycolsep}{12pt}\renewcommand{\arraystretch}{2.5}
			\begin{array}{llcccc}
			\lambda(0) & \lambda(i), i \geq 1 & R_{\lambda}^{*}(t) & \xi_{\lambda}^{}=R'_{\lambda}(1) & \mathrm{Mult}(\xi_{\lambda}) \\
			\hline
			\tiny{\raisebox{.2em}{\begin{ytableau} 0 & 1 & 2 \end{ytableau}}} & & (t+r-1)(t+2r-1)(t+3r-1) & 18r^3-11r^2 & 1 \\
			\tiny{\raisebox{.2em}{\begin{ytableau} 0 \\ $-1$ \\ $-2$ \end{ytableau}}} & & (t-r-1)(t-1)(t+r-1) & r^2 & 1 \\
			\tiny{\raisebox{.2em}{\begin{ytableau} 0 \end{ytableau}}} & \tiny{\raisebox{.2em}{\begin{ytableau} 0 \\ -1 \end{ytableau}}} & (t-r-1)(t-1)(t+r-1) & r^2 & 9(r-1) \\
			& \tiny{\raisebox{.2em}{\begin{ytableau} 0 & 1 \\ -1 \end{ytableau}}} & (t-r-1)(t-1)(t+r-1) & r^2 & 4(r-1) \\
			\tiny{\raisebox{.2em}{\begin{ytableau} 0 \end{ytableau}}} & \tiny{\raisebox{.2em}{\begin{ytableau} 0 & 1 \end{ytableau}}} & (t+r-1)^2(t-1) & -r^2 & 9(r-1) \\
			\tiny{\raisebox{.2em}{\begin{ytableau} 0 & 1 \\ -1 \end{ytableau}}} & & (t-1)(t+r-1)(t+2r-1) & -2r^2 & 4 \\
			\tiny{\raisebox{.2em}{\begin{ytableau} 0 & 1 \end{ytableau}}} & \tiny{\raisebox{.2em}{\begin{ytableau} 0 \end{ytableau}}} & (t-1)(t+r-1)(t+2r-1) & -2r^2 & 9(r-1) \\
			& \tiny{\raisebox{.2em}{\begin{ytableau} 0 & 1 & 2 \end{ytableau}}} & (t-1)(t+r-1)(t+2r-1) & -2r^2 & r-1 \\
			& \tiny{\raisebox{.2em}{\begin{ytableau} 0 \\ $-1$ \\ $-2$ \end{ytableau}}} & (t-2r-1)(t-r-1)(t-1) & -2r^2 & r-1 \\

%
%
				\end{array}}
				\]
\end{table}
\end{center}

\begin{center}
\begin{table}
\caption{Zero distance, equivalently codimension, spectrum of $\Gamma(G(r,1,3),T)$} 
\label{zerodistspecG(r,1,3)}
				\[
		{\renewcommand{\arraycolsep}{12pt}\renewcommand{\arraystretch}{2.5}
			\begin{array}{llcccc}
			\lambda(0) & \lambda(i), i \geq 1 & R_{\lambda}^{*}(t) & R_{\lambda}(t) & \xi_{\lambda}^{}=R'_{\lambda}(1) & \mathrm{Mult}(\xi_{\lambda}) \\
			\hline
				
			& \scriptsize{\raisebox{.2em}{\begin{ytableau} 0 \end{ytableau}}\,,\,\raisebox{.2em}{\begin{ytableau} 0 & 1 \end{ytableau}}} & (t-1)^2(t+r-1) & (1-t)^2(1+(r-1)t) & 0 & 9(r-1)(r-2) \\
			\scriptsize{\raisebox{.2em}{\begin{ytableau} 0 \\ $-1$ \end{ytableau}}} & \scriptsize{\raisebox{.2em}{\begin{ytableau} 0 \end{ytableau}}} & (t-1)^2(t+r-1) & (1-t)^2(1+(r-1)t) & 0 & 9(r-1) \\
			\scriptsize{\raisebox{.2em}{\begin{ytableau} 0 \end{ytableau}}} & \scriptsize{\raisebox{.2em}{\begin{ytableau} 0 \end{ytableau}}\,,\,\raisebox{.2em}{\begin{ytableau} 0 \end{ytableau}}} & (t-1)^2(t+r-1) & (1-t)^2(1+(r-1)t) & 0 & 36{r-1 \choose 2} \\

			& \scriptsize{\raisebox{.2em}{\begin{ytableau} 0 \\ $-1$ \end{ytableau}}\,,\,\raisebox{.2em}{\begin{ytableau} 0 \end{ytableau}}} & (t-1)^2(t-r-1) & (1-t)^2(1-(r+1)t) & 0 & 9(r-1)(r-2) \\

			& \scriptsize{\raisebox{.2em}{\begin{ytableau} 0 \end{ytableau}}\,,\,\raisebox{.2em}{\begin{ytableau} 0 \end{ytableau}}\,,\,\raisebox{.2em}{\begin{ytableau} 0 \end{ytableau}}} & (t-1)^3 & (1-t)^3 & 0 & 36{r-1 \choose 3} \\
			\end{array}}
		\]
		\end{table}
\end{center}
\end{landscape}
\end{example}

\bibliographystyle{abbrv}
\bibliography{DistanceSpectra}

\end{document}